\documentclass[11pt,letterpaper]{article}
\usepackage[margin=1in]{geometry}
\usepackage[utf8]{inputenc} 
\usepackage[T1]{fontenc}    
\usepackage{dsfont}
\usepackage[maxfloats=100]{morefloats}[2015/07/22]
\usepackage{multirow}
\usepackage{wrapfig}
\usepackage[T1]{fontenc}
\usepackage{lmodern}
\usepackage{booktabs}       
\usepackage{amsfonts}       
\usepackage{nicefrac}       
\usepackage{microtype}      
\usepackage{enumerate}
\usepackage{lipsum}
\usepackage{mathtools}
\usepackage{cuted}
\usepackage{float}
\usepackage[dvipsnames,table,xcdraw]{xcolor}
\usepackage{bbm}
\usepackage{varioref}
\usepackage{hyperref}
                                
\usepackage{amssymb} 
\usepackage{rotating}
\usepackage{graphicx}
\usepackage{subcaption}
\usepackage[normalem]{ulem}

\usepackage{amsthm}
\DeclareUnicodeCharacter{00A0}{~}
\usepackage{algorithm}
\usepackage{algpseudocode}
\algnewcommand{\Initialize}[1]{%
	\State \textbf{Initialization:}
	\Statex {\raggedright #1}
}
\newtheorem{assumption}{Assumption}
\newtheorem{theorem}{Theorem}
\newtheorem{lemma}{Lemma}
\newtheorem{proposition}{Proposition}
\newtheorem{definition}{Definition}
\newtheorem{corollary}{Corollary}
\theoremstyle{plain}
\newtheorem{remark}{Remark}

\definecolor{ao}{rgb}{0.0, 0.5, 0.0}
\newcommand{\afj}[1]{{\color{black}#1}} 
\newcommand{\za}[1]{{\color{black}#1}} 
\newcommand{\fy}[1]{{\color{black}#1}}
\newcommand{\fyy}[1]{{\color{black}#1}}
\newcommand{\aj}[1]{{\color{black}#1}}
\newcommand{\af}[1]{{\color{black}#1}}
\newcommand{\afz}[1]{{\color{black}#1}}
\newcommand{\ze}[1]{{\color{black}#1}}
\newcommand{\zal}[1]{{\color{black}#1}}
\newcommand{\zz}[1]{{\color{black}#1}}
\usepackage{pifont}
\newcommand{\cmark}{\ding{51}}%
\newcommand{\xmark}{\ding{55}}%
\usepackage[textsize=small]{todonotes}

\usepackage{tcolorbox}

\usepackage{amsmath, amssymb, bbm, xspace}

\def\spose#1{\hbox to 0pt{#1\hss}}

\def\text #1{\hbox{\quad#1\quad}}

\def\nthinsp{\mskip -2   mu}

\def\superstar{^{\raise 0.5pt\hbox{$\nthinsp *$}}}
\def\SUPERSTAR{^{\raise 0.5pt\hbox{$*$}}}

\def\lamstarT {\lambda^{\raise 0.5pt\hbox{$\nthinsp *$}T}}

\def\hbar{\skew{4.2}\bar h}

		\def\bk1{{\rm 1\kern-.17em l}}
		\def\bkD{{\rm I\kern-.17em D}}
		\def\bkR{{\rm I\kern-.17em R}}
		\def\bkP{{\rm I\kern-.17em P}}
		\def\bkY{{\bf \kern-.17em Y}}
		\def\bkZ{{\bf \kern-.17em Z}}

		\def\beq{\begin{eqnarray}}
		\def\bc{\begin{center}}
		\def\be{\begin{enumerate}}
		\def\bi{\begin{itemize}}
		\def\bs{\begin{small}}
		\def\bS{\begin{slide}}
		\def\ec{\end{center}}
		\def\ee{\end{enumerate}}
		\def\ei{\end{itemize}}
		\def\es{\end{small}}
		\def\eS{\end{slide}}
		\def\eeq{\end{eqnarray}}

	\def\cp2problem#1#2#3#4{\fbox
		 {\begin{tabular*}{0.9\textwidth}
			{@{}l@{\extracolsep{\fill}}l@{\extracolsep{6pt}}l@{\extracolsep{\fill}}c@{}}
				#1 & & $#4 $
			\end{tabular*}}}

		\renewcommand{\emph}[1]{\textbf{#1}}

		\def\bk1{{\rm 1\kern-.17em l}}
		\def\bkD{{\rm I\kern-.17em D}}
		\def\bkR{{\rm I\kern-.17em R}}
		\def\bkP{{\rm I\kern-.17em P}}
		
		\def\bkZ{{\bf{Z}}}

\newcommand {\beeq}[1]{\begin{equation}\label{#1}}
\newcommand {\eeeq}{\end{equation}}
\newcommand {\bea}{\begin{eqnarray}}
\newcommand {\eea}{\end{eqnarray}}

\def\texitem#1{\par\smallskip\noindent\hangindent 25pt
               \hbox to 25pt {\hss #1 ~}\ignorespaces}

\title{\LARGE \bf Randomized Lagrangian Stochastic Approximation for Large-Scale Constrained Stochastic Nash Games}

\author{Zeinab Alizadeh$^{*}$%
\and
 Afrooz Jalilzadeh\thanks{Department of Systems and Industrial Engineering,
        University of Arizona, Tucson, Arizona 85721, USA.
       {\tt\small zalizadeh@arizona.edu} and {\tt\small afrooz@arizona.edu}}%
\and Farzad~Yousefian\thanks{Department of Industrial and Systems Engineering, Rutgers University, Piscataway, NJ 08854, USA. {\tt\small farzad.yousefian@rutgers.edu}}}
 \date{}

\begin{document}
\sloppy
\maketitle
\thispagestyle{empty}
\pagestyle{plain}

\maketitle
\begin{abstract}
In this paper, we consider stochastic monotone Nash games where each player's strategy set is characterized by possibly a large number of explicit convex constraint inequalities. Notably, the functional constraints of each player may depend on the strategies of other players, allowing for capturing a subclass of generalized Nash equilibrium problems (GNEP). While there is limited work that provide guarantees for this class of stochastic GNEPs, even when the functional constraints of the players are independent of each other, the majority of the existing methods rely on employing projected stochastic approximation (SA) methods. However, the projected SA methods perform poorly when the constraint set is afflicted by the presence of a large number of possibly nonlinear functional inequalities. Motivated by the absence of performance guarantees for computing the Nash equilibrium in constrained {stochastic monotone Nash} games, we develop a single timescale randomized Lagrangian multiplier stochastic approximation method where in the primal space, we employ an SA scheme, and in the dual space, we employ a randomized block-coordinate scheme where only a randomly selected Lagrangian multiplier is updated. We show that our method achieves a convergence rate of $\mathcal{O}\left(\frac{\log(k)}{\sqrt{k}}\right)$ for suitably defined suboptimality and infeasibility metrics in a mean sense.
\end{abstract}

\section{Introduction}\label{sec:intro}
{Noncooperative game theory provides a mathematical framework to study multi-agent decision making problems that have emerged in a wide range of applications including electricity markets~\cite{hu2007using}, transportation networks~\cite{ferris1997engineering}, and signal processing~\cite{deligiannis2017game}, among many others. While the multidisciplinary field of game theory finds its origin in the work by von Neumann and Morgenstern~\cite{von1947theory}, the notion of a Nash equilibrium (NE) was introduced and its existence was provably shown by John Nash~\cite{nash1951non}. Noncooperative Nash game is a modeling framework where a finite collection of selfish agents compete with each other and seek to optimize their own individual objectives. Such a competition is often subject to limited resources characterized by functional constraints. In this work, our primary focus lies in computing an NE for large-scale constrained Nash game formulations afflicted by the presence of uncertainty in the objectives of the agents. More precisely, we consider stochastic monotone Nash games with a large number of (possibly nonlinear) functional constraints described as follows. Let $N \geq 1$ denote the number of players. For all $i =1,\ldots, N$, the $i$th player is associated with the following constrained stochastic optimization problem.
 \begin{tcolorbox}
    \vspace{-0.1in}
\begin{align}\label{prob:snash_nlpc}  
\min_{x_{i}\in \mathcal{X}_i } \qquad & H_i(x)\triangleq\mathbb{E}[h_i(x_i,x_{-i},\xi)]  \tag{P$_i(x_{-i})$}\\
\hbox{where}\qquad &  \mathcal{X}_i\triangleq \left\{x_i \in X_i {\ \subseteq \mathbb{R}^{n_i}}\mid g_{i,\ell}(x_i{,x_{-i}}) \leq 0,\ \ \hbox{for all } \ell=1,\dots,J_i\right\}\ \   \notag
\end{align}
\end{tcolorbox}

\noindent where $x_i \in \mathbb{R}^{n_i}$ denotes the strategy of the $i$th player, $x_{-i} \in \mathbb{R}^{n-n_i}$ is the collection of the strategies of the other players, $n \triangleq \sum_{i=1}^N n_i$, $h_i:\mathbb{R}^{n}\times \mathbb{R}^d \to \mathbb{R}$ denotes the stochastic cost function associated with the $i$th player. The uncertainty in the game is characterized by the random variable $\xi: \Omega\to\mathbb{R}^d$ associated with the probability space $(\Omega, \mathcal{F},\mathbb{P})$. The constraint set of the $i$th player is expressed in terms of explicit convex constraint inequalities in terms of the jointly convex functions $g_{i,\ell}:\mathbb{R}^{n}\to \mathbb{R}$, for all $\ell = 1,\ldots,J_i$. The $i$th player's strategy is a subset of a nonempty convex set denoted by $X_i \subseteq \mathbb{R}^{n_i}$.  While we will provide the detailed description of our assumptions in subsequent sections, it is worth emphasizing that throughout, we assume that all the aforementioned functions are merely convex. 

Problem~\eqref{prob:snash_nlpc} is a subclass of the generalized Nash equilibrium problems (GNEP) that have been extensively employed in the literature in formulating applications arising in economics and operations research, among others~\cite{facchinei2010generalized,KRILASEVIC2023110931}. Recall that in GNEPs, players seek the NE by {\it simultaneously} satisfying the constraints. This is different from other classes of games where players make decisions in a specific order, e.g., in Stackelberg games. 

{Note that a} popular subclass of the problem~\eqref{prob:snash_nlpc} is the stochastic minimax problem. Consider the following stochastic merely-convex-merely-concave minimax optimization problem with possibly many functional constraints. 
\begin{tcolorbox}
    \vspace{-0.1in}
\begin{align}\label{prob:Sminimax_nlpc} 
\min_{u \in \mathcal{U}}  \max_{v \in   \mathcal{V}}\qquad & H(u,v)\triangleq\mathbb{E}[h(u,v,\xi)]  \\
\hbox{where}\qquad &  \mathcal{U}\triangleq \left\{u \in U \mid g_{1,\ell}(u) \leq 0,\ \ \hbox{for all } \ell=1,\dots,J_1\right\}\ \  \hbox{and}\ \  U \subseteq \mathbb{R}^{n_1},\notag \\
&\mathcal{V}\triangleq \left\{v \in V \mid g_{2,\ell}(v) \leq 0,\ \ \hbox{for all } \ell=1,\dots,J_2\right\}\ \  \hbox{and}\ \  V \subseteq \mathbb{R}^{n_2}.\notag
\end{align}
\end{tcolorbox}

Minimax optimization can indeed be viewed as a subclass of two-person zero-sum games. The existence of equilibrium in such a game is established by the celebrated von Neumann's minimax theorem in 1928 \cite{minimax28} that appears amongst the most fundamental results in game theory. The research on the development of gradient-type methods for solving minimax problems, also known as the problem of finding {\it saddle points}, dates back to as early as 1970s, including the work by Korpelevich~\cite{korp76} and Golshtein~\cite{golshtein1974generalized}, followed by efforts on on the development  gradient descent ascent as well as primal-dual methods (e.g., see ~\cite{chen97,Nem04,ned09,zhao2022accelerated,hamedani2021primal} and \cite[Chp. 1]{facchinei02finite}). More recently, minimax problems have drawn an increasing attention in areas including adversarial learning \cite{GAN14,NEURIPS2018_5a9d8bf5,sinha2018certifiable}, fairness in machine learning \zz{\cite{fairgan18,jin2022sharper}}, and distributionally robust federated learning \cite{distrobustfl20}, to name a few.

\noindent \textbf{Existing methods and research gap.} In addressing deterministic games, iterative methods for approximating an equilibrium find their origin in 1960s in the seminal work by Scarf~\cite{scarf1967approximation} (see~\cite[Chapter~12]{facchinei02finite} for a detailed review of deterministic methods). The prior algorithmic efforts in addressing stochastic Nash games, however, find their roots in the work by Jiang and Xu~\cite{jiangxu2008} in 2008, where a stochastic approximation (SA) method was developed for addressing stochastic variational inequality (VI) problems with strongly monotone and Lipschitzian mappings. Recall that given a set $\mathcal{X}$ and a single-valued mapping $F:\mathbb{R}^n \to \mathbb{R}^n$, vector $x \in X$ solves $\mbox{VI}(\mathcal{X},F)$ if $F(x)^T(y-x)\geq 0$ for all $y \in \mathcal{X}$. Under some mild convexity and differentiability assumptions, it can be shown that~\cite[Chapter~1]{facchinei02finite} the set of equilibria of the stochastic game~\eqref{prob:snash_nlpc}, for $i=1,\ldots,N$, is characterized by the solution set of $\mbox{VI}(\mathcal{X},F)$ where $\mathcal{X}\triangleq \prod_{i=1}^N \mathcal{X}_i$ and $F(x) \triangleq \left(\nabla_{x_1} \mathbb{E}[h_1(x,\xi)];\ldots;\nabla_{x_N} \mathbb{E}[h_N(x,\xi)] \right)$. In view of this result, seeking a Nash equilibrium of a stochastic game is equivalent to solving the aforementioned stochastic VI. The convergence and rate analysis of SA schemes for solving VIs under weaker monotonicity and smoothness assumptions were studied more recently in works including~\cite{KannanShanbhag2012,KoshalNedichShanbhag2013,FarzadMathProg17}. Also, stochastic extragradient methods and their variance-reduced variants were studied in~\zz{\cite{FarzadSetValued18,long2023fast,iusem2017extragradient}}. 
\begin{table}
\centering
\caption{Solution methods with rate statements for variational inequality problems}
\begin{tabular}{|c|c|c|c|c|c|}
\hline
Ref. &  Problem  &Rate  & Nonlinear const.\\ \hline
{Proximal Extra-Gradient}\cite{doi:10.1080/10556788.2017.1300899} & VI  &$\mathcal{O}(1/\epsilon)$ &\xmark  \\ \hline
 {SMP}\cite{juditsky2011solving} &SVI  &$\mathcal{O}(1/\epsilon^2)$&\xmark\\ \hline
{DS-SA} \cite{iusem2019variance}&SVI  &$\mathcal{O}(\log (1/\epsilon)/\epsilon^2)$&\xmark\\ \hline
 {RLSA} {(This paper)} &{SVI}   &{$\mathcal{O}(\log (1/\epsilon)/\epsilon^2)$} & {\cmark} \\ \hline
\end{tabular}

\label{tab:1}
\end{table}
Despite these advances, it is often assumed in the above-mentioned methods that the sets $\mathcal{X}_i$ is easy-to-project on and accordingly, the algorithmic framework in these works relies on projected schemes. However, in the following cases, $\mathcal{X}_i$ may become difficult-to-project on: (i) When the dimensionality of the solution space, i.e., $n$, is large; (ii) When the number of the constraints is large. For example, in the game setting $\sum_{i=1}^N J_i$ could be large; (iii) The constraint set may be characterized by nonlinear constraints. In fact, we are unaware of any iterative methods with provable complexity guarantees for the resolution of even deterministic variants of constrained monotone Nash games. Our research in this paper is precisely motivated by this shortcoming in the literature. 
\begin{table}
\centering
\caption{A subset of methods with guarantees for saddle point problems}
\begin{tabular}{|c|c|c|c|c|}
\hline
Ref. & Stoch.&Non-bilinear & Convex  & Nonlinear const.\\ \hline
PDHG \cite{chambolle2016ergodic},Acc-SP-HPE\cite{he2016accelerated} &\xmark & \xmark & $\mathcal{O}(1/\epsilon)$  &\xmark \\ \hline
 Acc-HPE-type \cite{kolossoski2017accelerated},SMP\cite{juditsky2011solving} &\xmark & \cmark & $\mathcal{O}(1/\epsilon)$  &\xmark\\ \hline

Acc- BD\cite{he2015accelerating} & \cmark &\xmark & $\mathcal{O}(1/\epsilon)$  &\xmark\\ \hline
SAA\cite{nemirovski2009robust},SADMM\cite{zhao2019optimal} &\cmark&\cmark&$\mathcal{O}(1/ \epsilon^2)$&\xmark\\ \hline
{RLSA(This paper)} &\cmark& {\cmark} & {$\mathcal{O}(\log (1/\epsilon)/\epsilon^2)$} & {\cmark} \\ \hline
\end{tabular}

\label{tab:2}
\end{table}

\noindent \textbf{Main contributions.} In Table~\ref{tab:1} and Table~\ref{tab:2}, we provide a summary of the main results in our work and we compare them with some of the existing methods for addressing monotone VIs and minimax problems. To highlight our contributions, we first provide a brief review of some of the existing avenues for addressing monotone Nash games and VIs with explicit constraints. The duality theory for VIs and the notion of the dual VI has been studied by Mosco~\cite{mosco1972dual} in 1972 which was later improved in~\cite{gabay1983chapter,eckstein1999smooth}. Extending the duality framework devised in~\cite{auslender1976optimisation}, Auslender and Teboulle~\cite{AML_VI_2000_siopt} developed a Lagrangian duality scheme for solving multi-valued variational inequality problems with maximal monotone operators and explicit convex constraint inequalities. Leveraging entropic proximal terms, interior proximal point methods were developed for solving constrained VIs in works including~\cite{auslender1995interior,burachik1998generalized}. Although the aforementioned dual-based methods are endowed with asymptotic convergence guarantees, convergence speed of Lagrangian dual methods for solving constrained VIs is not known. In particular, we are interested in investigating whether it is possible to devise suitable Lagrangian dual methods that can be guaranteed with convergence speeds of similar order of magnitude to those of primal-dual methods developed for standard constrained optimization methods~\cite{yangyang_2020_PD_large_constraints}. We show that this is indeed possible. We summarize our main contributions in the following.

\noindent (i) {\it A single timescale randomized primal-dual stochastic approximation method.} Leveraging the primal-dual framework for addressing constrained stochastic optimization problems, we devise a randomized primal-dual stochastic approximation method for solving VIs with merely monotone and stochastic mappings with explicit constraint inequalities. To capture large-scale constrained stochastic Nash games, we employ a randomized block scheme for updating the Lagrange multipliers. Importantly, this scheme is single timescale and efficient to implement. 

\noindent (ii) {\it New convergence rate statements.} In contrast with standard optimization problems, one of the main challenges in addressing VIs lies in the lack of availability of suitable error metrics that rely on objective function values. In particular, this challenge introduces some difficulty in the convergence rate analysis of monotone VIs, an issue that is exacerbated in the presence of explicit constraint inequalities. Motivated by earlier efforts~\cite{FarzadMathProg17,KaushikYousefianSIOPT2021}, leveraging the notion of dual gap functions, we analyze the convergence of the proposed method and derive convergence rates of $\mathcal{O}\left(\frac{\log(k)}{\sqrt{k}}\right)$ for both suboptimality and infeasibility metrics in a mean sense.


\noindent \textbf{Outline of the paper.} The remainder of the paper is organized as follows. In Section~\ref{sec:prelim}, we provide the main assumptions and review some preliminary results that are employed in the analysis. In Section~\ref{sec:alg}, we present the outline of the proposed algorithm along with some definitions. In Section~\ref{sec:rate} we establish convergence properties of the method and derive explicit performance guarantees. We present some concluding remarks in Section~\ref{sec:conc}. Lastly, Section~\ref{sec:app} includes the proofs for some of the results in the paper. 

}
\section{Preliminaries}\label{sec:prelim}
\fy{To address the stochastic game~\ref{prob:snash_nlpc} for $i \in [N]$, we consider the stochastic VI problem described as follows.
\begin{tcolorbox}
    \vspace{-0.2in}
\begin{align}\label{prob:SVI_nlp}\tag{\bf cSVI}
& \hbox{Find } x \in \mathcal{X} \quad \hbox{such that } \quad \mathbb{E}[F(x,\xi)]^T(y-x) \geq 0, \quad \hbox{for all } y \in \mathcal{X}\\
& \hbox{where}\qquad   \mathcal{X}\triangleq \left\{x \in X \mid f_j(x) \leq 0,\ \ \hbox{for all } j=1,\dots,J\right\}\ \  \hbox{and}\ \  X \triangleq \prod\nolimits_{i=1}^NX_i.\notag
\end{align}
\end{tcolorbox}
The details of our assumptions on the mapping $F$, functions $f_j$, and sets $X_i$ are provided as follows. 
\begin{assumption}[Problem properties]\label{assum:problem} \em {Consider problem \eqref{prob:SVI_nlp}.} Let the following holds.

\noindent \fy{(i)} Mapping $F(\bullet):\mathbb{R}^n \to \mathbb{R}^n$ is real-valued, continuous, and merely monotone on its domain, {i.e. $\langle F(x)-F(y),x-y\rangle\geq 0,$ for all $x,y\in X$.}

\noindent \fy{(ii)} Function $\fy{f_j (\bullet)}:\mathbb{R}^n \to \mathbb{R}$ is real-valued, merely convex on its domain for all $j=1,\ldots,J$. 

\noindent \fy{(iii)} Set $X \subseteq \mbox{int}\left(\mbox{dom}(F)\cap(\cap_{j=1}^J\mbox{dom}(f_j))\right)$ is nonempty, compact, and convex. 

\noindent \afz{(iv) The Slater condition holds, i.e., there exists $\hat x\in X$ such that $f_j(\hat x)<0$ for all $j=1,\hdots,J$.}
\end{assumption}
\begin{remark}\em 
Note that problem \eqref{prob:SVI_nlp} captures the stochastic game \ref{prob:snash_nlpc}. In fact, given the objective functions $h_i(\bullet, \xi)$ and constraint functions $g_{i,\ell}$ in \ref{prob:snash_nlpc}, $x$ is an NE if and only if $x$ solves \eqref{prob:SVI_nlp} where $f_j(x) \triangleq g_{i,\ell}(\fyy{x})$ where $j :=\ell+ \sum_{t=1}^{i-1}J_t$ for $\ell \in [J_i]$. 
\end{remark}
}


\begin{definition}[Augmented-Lagrangian function]\label{def:lagrangian_function}
Given $x,y \in \mathbb{R}^n$, $\lambda \in \mathbb{R}^J$, and $\rho>0$, we define
\begin{align*}
&\mathcal{L}_\rho(x,y,\lambda) \triangleq F(y)^T(x-y)+\Phi_\rho(x,\lambda),\\
\hbox{where } \quad & \Phi_\rho(x,\lambda)\triangleq\tfrac{1}{J}\sum_{j=1}^J\phi_\rho(f_j(x),\aj{\lambda^{(j)}}) \ \ \hbox{and} \ \ 
 \phi_\rho(u,v) \triangleq   \left \{
  \begin{aligned}
    &uv+\tfrac{\rho}{2}u^2, && \text{if}\ \rho u +v \geq 0,\\
    &-\tfrac{v^2}{2\rho}, && \text{otherwise.}
  \end{aligned} \right.
\end{align*}

  \end{definition}

\afz{Similar to the traditional constrained optimization techniques, the nonlinear constrains in problem \eqref{prob:SVI_nlp} can be combined with the objective function using some multipliers. Using this technique we can characterize the optimality condition of problem \eqref{prob:SVI_nlp} in the following result.}

{\begin{proposition}[Karush--Kuhn--Tucker (KKT) conditions]\label{assum:kkt}\em 
\afz{Consider problem \eqref{prob:SVI_nlp} and suppose Assumption \ref{assum:problem} holds. Let $f(x)\triangleq (f_1(x),\ldots,f_J(x))^T$ and the gradient matrix $\nabla f(x) \triangleq (\nabla f_1(x),\ldots,\nabla f_J(x))^T \in \mathbb{R}^{n\times J}$. There exists $x^* \in \mathbb{R}^n$ and $\lambda^* \in \mathbb{R}^J$ satisfying the following KKT conditions:} 

\noindent (i)  $0 \in F(x^*) +J^{-1}\nabla f(x^*)^T\lambda^*+ \mathcal{N}_X(x^*)$.

\noindent (ii)  $0\leq \lambda^* \perp -f(x^*) \geq 0$.

\noindent (iii) $x^* \in X$.

\end{proposition}

\begin{proof}
\afz{Note that any solution $x$ of \eqref{prob:SVI_nlp} is also a solution of the following optimization problem:
\begin{align}\label{opt_VI}
& \min_{y\in X }\ y^TF(x)\\
 \nonumber&\mbox{s.t.}\ f_j(y)\leq 0\quad \forall j.
 \end{align}
 Since the Slater condition holds, the first-order KKT condition for \eqref{opt_VI} implies that there exists 
$x^* \in \mathbb{R}^n$ and $\lambda^* \in \mathbb{R}^J$ satisfying conditions (i)-(iii).}
\end{proof}}
We will utilize the following definition in the convergence and rate analysis. 
\begin{definition}\label{def:gap}\em 
Consider the $\mbox{VI}(\mathcal{X},F)$ where $X$ is a closed convex set and $F$ is a real-valued {monotone} map. The dual gap function $\mbox{Gap}^*:\mathcal{X} \to \mathbb{R}\cup \{+\infty\}$ is defined for any $x \in \mathbb{R}^n$ as 
\begin{align}\label{eqn:gap}
\mbox{Gap}^*(x)\triangleq \sup_{y\in \mathcal{X}} F(y)^T(x-y).
\end{align}
\end{definition}
\begin{remark}\em 
Note that by the definition, $\mbox{Gap}^*(x) \geq 0$ for all $x \in \mathcal{X}$. Also, under some mild conditions, $\mbox{Gap}^*(x) = 0$ implies that $x$ is a solution to $\mbox{VI}(\mathcal{X},F)$. This is formally stated below.
\end{remark}
\aj{\begin{remark}\em 
Karamardian  \cite{karamardian1976existence} showed that under continuity and pseudomonotonicity of the
operator $F$, solving \eqref{prob:SVI_nlp} problem is equivalent to solving Minty stochastic variational inequality (MSVI) \cite{minty1962monotone} problem. Such a problem requires an $x^* \in X$ such that 
\begin{align}\tag{MSVI}
 (x^*-x)^TF(x)\leq 0, \qquad \hbox{for all } x\in \mathcal X. \end{align}
Therefore, to obtain the convergence rate we adopt the {dual} gap function. 
Note that {$\mbox{Gap}^*(\bullet)$} is well-defined when $\mathcal X$ is a compact set, that follows from Assumption \ref{assum:problem}~(iii). 
\end{remark}}
\afj{By invoking Proposition \ref{assum:kkt} and Assumption \ref{assum:problem}, we can establish the following two results for problem \eqref{prob:SVI_nlp}. These results will be employed later to demonstrate the boundedness of dual iterates and to obtain convergence rate results. We have provided the proofs of the following lemmas in the appendix.}
{\begin{lemma}\label{lem:err_func}\em
Consider problem \eqref{prob:SVI_nlp} under Assumption \ref{assum:problem}. Then for any primal-dual solution pair $(x^*,\lambda^*),$ the following holds
{\begin{align*} 
F(x^*) ^T(x-x^*) +J^{-1}f(x)^T\lambda^*  \geq 0, \qquad \hbox{for all } x \in X.  
\end{align*}}
\end{lemma}
}
{\begin{lemma}\label{lem main}\em
Consider problem \eqref{prob:SVI_nlp}. Let Assumption \ref{assum:problem} holds. Assume that for any $x \in X$ and $\lambda \in \mathbb{R}_+^J$ we have
\begin{align}\label{ineq:decomp_lemma_assump_ineq}
 {F(x)^T}(\hat x-x) +J^{-1}f(\hat x)^T\lambda \leq \Phi_\rho(x,\hat \lambda)+{C(x,\lambda)},
\end{align}
where $\hat x \in X$ and $\hat \lambda \in \mathbb{R}_{+}^J$ are arbitrary vectors. Then for any primal-dual solution pair $(x^*,\lambda^*),$ the following holds.

\noindent (i) $J^{-1} \mathbf{1}^T[f(\hat x)]_+ \leq {C(x^*,\tilde \lambda)} $, where for all $j\in [J]$ we define $$\tilde \lambda_j\triangleq \begin{cases} 1+\lambda^*_j, &\text{if $f_j(\hat x)>0$,}\\ 0, & \text{otherwise.}\end{cases}$$

\noindent (ii) {$\sup_{x\in \mathcal X}\{F(x)^T(\hat x-x)\}\leq {\sup_{x\in \mathcal X}\{C(x,0)\}}$.}
\end{lemma}
}
\section{Algorithm outline}\label{sec:alg}
{The outline of the proposed method is presented by Algorithm~\ref{alg:primal_dual_SVI}. The sequence of the primal iterates is denoted by $\{x_k\}$ and the sequence of the dual iterates is denoted by $\{\lambda_k\}$. This is a single timescale Lagrangian stochastic approximation scheme that includes two main steps. At each iteration, in the dual step in equation \eqref{eqn:dual_step}, a randomly selected dual variable $\lambda^{(j)}$ is updated, while in the primal step in equation \eqref{eqn:primal_step}, the primal variables are updated. The stepsize sequence is denoted by $\{\gamma_k\}$ and the penalty sequence is denoted by $\{\rho_k\}$. In addition to the primal and dual variables that are updated at each iteration, both the stepsize and penalty parameter are updated iteratively. Our goal in this work lies in proving that Algorithm~\ref{alg:primal_dual_SVI} can be employed for solving the stochastic VI problem \eqref{prob:SVI_nlp} where the constraint set is characterized by explicit functional constraints. This result will be presented in the next section by Theorem~\ref{thm:rates} where we provide specific update rules for both $\gamma_k$ and $\rho_k$ such that the convergence of the proposed method can be guaranteed and non-asymptotic convergence rates can be derived. Before we proceed with the analysis of the method, we provide some definitions that will be utilized. 

\begin{remark}
Note that the Augmented Lagrangian function introduced in Definition~\ref{def:lagrangian_function} can be viewed as a relaxed variant of the following standard Augmented Lagrangian function of the form
\begin{align*}
\mathcal{L}_\rho(x,\lambda) & \triangleq \sup_{y \in \mathcal{X}}\mathcal{L}_\rho(x,y,\lambda) = sup_{y \in \mathcal{X}} \{F(y)^T(x-y)\}+\Phi_\rho(x,\lambda)\\
& = \mbox{Gap}^*(x) +\Phi_\rho(x,\lambda).
\end{align*}
Indeed, one of the key challenges in employing the Augmented Lagrangian function $\mathcal{L}_\rho(x,\lambda)$ is the presence of the supremum and nondifferentiability of the dual gap function. Further, even when the samples $F(\bullet,\xi_k)$ are unbiased, the standard Augmented Lagrangian function above may be biased, due to the presence of the supremum which again, renders an issue in utilizing this Augmented Lagrangian function. To circumvent these challenges, we employ the relaxed variant of the Augmented Lagrangian function introduced in Definition~\ref{def:lagrangian_function}. Importantly, as it will be shown in Theorem~\ref{thm:rates}, utilizing the relaxed variant of the Augmented Lagrangian function allows us to derive the rate statements. This is indeed a key novelty in the design of the proposed method in this work.  
     \end{remark}
Throughout, we let the history of the method be denoted by $\mathcal{F}_k\triangleq \cup_{t=0}^{k-1}\{\xi_t,j_t\}$ for any $k\geq 1 $, and $\mathcal{F}_0\triangleq \{\xi_0,j_0\}$. 
\begin{algorithm}[t]
	\caption{Randomized Lagrangian stochastic approximation method {(RLSA)}}\label{alg:primal_dual_SVI}
{\begin{algorithmic}[1] 
	 \State \textbf{input}: Choose $x_0 \in X$, $\lambda_0 :=0_J$, and $\rho_0>0$ 
		\For {$k = 0,1, 2,\dots$}
		\State  Generate a random variable $j_k$ uniformly drawn from $\{1,\ldots,J\}$
	     \State Generate a random realization of $\xi$ denoted by $\xi_k$ and evaluate $F(x_k,\xi_k)$
		\State Update the dual variable $\lambda_k$ for all $j=1,\ldots,J$ as follows. 
		\begin{align}\label{eqn:dual_step}
		\lambda_{k+1}^{(j)}:=  \left \{
  \begin{aligned}
    &\left[\rho_k f_{j}(x_k)+\lambda_k^{(j)}\right]_+, && \text{if}\ j=j_k\\
    &\lambda_k^{(j)}, && \text{otherwise}
  \end{aligned} \right. 
		\end{align}
      \State Evaluate  $\tilde \nabla f_{j_k}(x_k) \in \partial f_{j_k}(x_k)$
		\State Update the primal variable $x_k$ as follows.
		\begin{align}\label{eqn:primal_step}
		x_{k+1}:=\Pi_X\left[x_k-\gamma_k\left(F(x_k,\xi_k)+\lambda_{k+1}^{(j_k)}\tilde \nabla f_{j_k}(x_k) \right)\right]
		\end{align}
	 \EndFor
	\end{algorithmic}}
\end{algorithm}

  {\begin{assumption}[Random samples]\label{assum:samples}\em Let the following holds.

\noindent (i) Samples $\xi_k$ are generated independently from the probability distribution of $\xi$ for $k\geq 0$.

\noindent (ii) Samples $j_k$, for $k\geq 0$, are generated independently from a uniform probability distribution such that $\mbox{Prob}(j_k=j)=J^{-1}$ for all $j=1,\ldots,J$.

\noindent (iii) Samples $\xi_k$ and $j_k$ are generated independently from each other. 

\noindent (iv)  $\mathbb{E}[F(x,\xi_k)-F(x)\mid x] =0$ for all $x \in X$ and all $k\geq 0$.

\noindent (v) There is some $\nu>0$ such that $\mathbb{E}[\|F(x,\xi_k)-F(x)\|^2\mid x] \leq \nu^2$ for all $x \in X$ and all $k\geq 0$.  
\end{assumption}}

{\begin{remark}\label{rem:bounds}\em
In view of Assumption \ref{assum:problem}, the subdifferential set $\partial f_j(x)$ is nonempty for all $x \in \mbox{int}(\mbox{dom}(f_j))$ and all $j=1,\ldots,J$. Also,  $f_j$ has bounded subgradients over $X$. Throughout, we let scalars $D_X$ and $D_f$ be defined as $D_X\triangleq \sup_{x \in X} \|x\|$ and $D_f\triangleq \max_{j \in [J]}\sup_{x \in X} |f_j(x)|$, respectively. Also, we let $C_F>0$ and $C_f>0$ be scalars such that $\|F(x)\|\leq C_F$ and $\|\tilde \nabla f_j(x)\|\leq C_f$ for all $\tilde \nabla f_j(x)\in \partial f_j(x)$, for all $x \in X$. 
\end{remark}}

\begin{definition}[Stochastic errors]\label{def:stoch_err}\em Let us define the following stochastic terms for $k\geq 0$.


\noindent {(i)} $w_k\triangleq F(x_k,\xi_k)-F(x_k)$.

\noindent {(ii)}  $\delta_k\triangleq \left[{\rho_k} f_{j_k}(x_k)+\lambda_k^{(j_k)} \right]_+\tilde \nabla f_{j_k}(x_k)-\tfrac{1}{J}\textstyle\sum_{j=1}^J	 \left[{\rho_k} f_{j}(x_k)+\lambda_k^{(j)} \right]_+\tilde \nabla f_{j}(x_k)$. 
\end{definition}}

\afz{In the next lemma, we show that the stochastic errors defined above are unbiased and have bounded variance. The proof is provided in the appendix.}

{\begin{lemma}[Properties of stochastic errors]\label{lem:prop_stoch_err}\em 
Consider Definition \ref{def:stoch_err}. Let Assumption \ref{assum:samples} holds. Then:

\noindent (i) ${ \mathbb{E}[w_k \mid \mathcal{F}_k]=0}$ {and $\mathbb{E}[\|w_k\|^2 \mid \mathcal{F}_k]\leq \nu^2$.}

\noindent (ii) $\mathbb{E}[\delta_k \mid \mathcal{F}_k]=0$ and $\mathbb{E}[\|\delta_k\|^2 \mid \mathcal{F}_k]\leq 2C_f^2\left(\rho_k^2D_f^2+\tfrac{\|\lambda_k\|^2}{J}\right).$
\end{lemma}
}

\begin{remark}\label{rem:stoch_err_subg}\em
\fy{Note that we have} $\tfrac{1}{J}\textstyle\sum_{j=1}^J	 \left[\af{\rho_k} f_{j}(x_k)+\lambda_k^{(j)} \right]_+\tilde \nabla f_{j}(x_k) \in \aj{\partial_x} \Phi_\af{\rho_k}(x_k,\lambda_k)$.  Throughout,  we use the following notation
\begin{align*}
\aj{\tilde{\nabla}_x} \Phi_\af{\rho_k}(x_k,\lambda_k) \triangleq \tfrac{1}{J}\sum_{j=1}^J	 \left[\rho_k f_{j}(x_k)+\lambda_k^{(j)} \right]_+\tilde \nabla f_{j}(x_k).
\end{align*}
\afz{Therefore, one can conclude that }
\begin{align*}
\|\aj{\tilde{\nabla}_x} \Phi_\af{\rho_k}(x_k,\lambda_k)\|^2 \leq \tfrac{1}{J}\sum_{j=1}^J	\left\| \left[\af{\rho_k} f_{j}(x_k)+\lambda_k^{(j)} \right]_+\tilde \nabla f_{j}(x_k)\right\|^2\leq 2\af{\rho_k}^2D_f^2C_f^2+\tfrac{2C_f^2}{J}\|\lambda_k\|^2.
\end{align*}
\end{remark}

\section{Convergence and rate analysis}\label{sec:rate}
To obtain the main results of this paper, we use the following technical lemmas. \afj{All related proofs are provided in the appendix.} 

\begin{lemma}\label{lem:error}\em
Given \fy{an} arbitrary sequences $\{\sigma_k\}_{k\geq 0}\subset \mathbb R^n$ and $\{\tau_k\}_{k\geq 0}\subset \mathbb R^{++}$, let $\{v_k\}_{k\geq0}$ be a sequence such that $v_0\in \mathbb R^n$ and $v_{k+1}=v_k+\tau_k\sigma_k$. Then, for all $k\geq 0$ and $x\in \mathbb R^n$,
$$\zal{ \sigma^T_{k}(x-v_k)} \leq {1\over 2\tau_k}\|x-v_k\|^2-{1\over 2\tau_k}\|x-v_{k+1}\|+{\tau_k\over 2}\|\sigma_k\|^2.$$
\end{lemma}

\aj{
\begin{lemma}\label{lem:bound_phi} \em{Consider Algorithm \ref{alg:primal_dual_SVI}.} Let $J_k^+=\{j\in[J]\mid \rho_k f_{j}(x_k)+\lambda_k^{(j)}\geq0\}$ and $J_k^-=[J] \backslash J_k^+$. Then, for any $\lambda { \in \mathbb{R}^J_{+}}$, the following holds:
\begin{align*}
&\quad -\Phi_{\rho_k}(x_k,\lambda_k)+{1\over J}\sum_{j=1}^J \lambda^{(j)}f_j(x_k)+{1\over 2\rho_k}\|\lambda_{k+1}-\lambda\|{^2}\\
&\leq {1\over 2\rho_k}\|\lambda_k-\lambda\|{^2}+\zal{ (\lambda_k-\lambda)^T (Je_{j_k}\odot \nabla_{\lambda} \Phi_{\rho_k}(x_k,\lambda_k)-\nabla_\lambda\Phi_{\rho_k}(x_k,\lambda_k))}+\Delta_k,
\end{align*}
where $\Delta_k{\triangleq}-\tfrac{1}{J}\sum_{j\in J_k^+}\tfrac{\rho_k}{2}(f_j(x_k))^2-\tfrac{1}{J}\sum_{j\in J_k^-}\tfrac{(\lambda_k^{(j)})^2}{2\rho_k}+\tfrac{1}{2\rho_k}\|\lambda_{k+1}-\lambda_k\|{^2}$.
\end{lemma}

\begin{lemma}\label{lem:bound_product} \em
Suppose Assumption \ref{assum:problem} holds. Then, the following holds:
\begin{itemize}
\item[(a)] $\|Je_{j_k}\odot \nabla_{\lambda} \Phi_\aj{\rho_k}(x_k,\lambda_k)\|^2\leq D_f^2.$
\item[(b)]  Let $\afj{\bar\sigma_k=Je_{j_k}\odot \nabla_{\lambda} \Phi_\aj{\rho_k}(x_k,\lambda_k)-\nabla_\lambda\Phi_\aj{\rho_k}(x_k,\lambda_k)}$ and $\{\bar v_k\}_{k\geq0}$ be a sequence such that $\bar v_0\in \mathbb R^n$ and $\bar v_{k+1}=v_k+\bar \tau_k\bar\sigma_k$ for some $\{\bar \tau_k\}_{k\geq0}$. Then, the following holds.
\begin{align*}&(\lambda_k-\lambda)^T (Je_{j_k}\odot \nabla_{\lambda} \Phi_\aj{\rho_k}(x_k,\lambda_k)-\nabla_\lambda\Phi_\aj{\rho_k}(x_k,\lambda_k))\\&\quad \leq \zal{(\bar v_k-\lambda_k)^T \bar \sigma_k}+{1\over 2\bar\tau_k}\|\lambda-\bar v_k\|^2-{1\over 2\bar\tau_k}\|\lambda-\bar v_{k+1}\|{^2}+{\bar \tau_k \afj{\|\bar \sigma_k\|^2}\over 2}.\end{align*}
\end{itemize}
\end{lemma}
}

\afz{Next, using Lemma \ref{lem:bound_phi} and \ref{lem:bound_product} we provide one-step analysis of our method by providing an upper bound on the reduction of the gap function in terms of the consecutive iterates.}

\begin{proposition}\label{prop:main_rec_ineq}\em
{Consider Algorithm \ref{alg:primal_dual_SVI}. Let Assumptions \ref{assum:problem} and \ref{assum:samples} hold. Then, for any $x \in X$ and $\lambda \in \mathbb{R}^J_{+}$ the following inequality holds.
\begin{align}\label{prop1}
\nonumber&(x_k-x)^TF(x)+J^{-1}f(x_k)^{T}\lambda- \Phi_\aj{\rho_k}(x,\lambda_k)\\&\nonumber
\quad\leq \tfrac{1}{2\gamma_k}\left(\|x_{k}-x\|^2-\|x_{k+1}-x\|^2\right)+\tfrac{1} {4\gamma_k}\left(\|x-v_k\|^2-\|x-v_{k+1}\|^2\right)\\
&\nonumber \qquad+\tfrac{1}{2\aj{\rho_k}}\left(\|\lambda_{k}-\lambda\|{^2}-\|\lambda_{k+1}-\lambda\|{^2}\right)+\tfrac{1}{2\bar\tau_k}\left(\|\lambda-\zz{\bar v_{k}}\|{^2}-\|\lambda-\zz{\bar v_{k+1}}\|{^2}\right)+2\gamma_kC_F^2\\
&\nonumber \qquad+4\gamma_kC_f^2\left(\aj{\rho_k^2}\ze{D_f^2}+\tfrac{1}{J}\|\lambda_k\|^2\right)+(v_k-x_k)^T(w_k+\delta_k)+2\gamma_k\|w_k+\delta_k\|^2\\&
\qquad + \zal{ (\bar v_k-\lambda_k)^T \bar \sigma_k} +\tfrac{\bar\tau_k\za{\|\bar \sigma_k\|^2}}{2}-\tfrac{1}{J}\sum_{j\in J_k^+}\tfrac{\aj{\rho_k}}{2}(f_j(x_k))^2-\tfrac{1}{J}\sum_{j\in J_k^-}\tfrac{(\lambda_k^{(j)})^2}{2\aj{\rho_k}},
\end{align}
where $J^{+}_k$ and $J^-_{k}$ are given by Lemma \ref{lem:bound_phi}.}
\end{proposition}
\begin{proof}
Let $x \in X$ and $\lambda \geq 0$ be arbitrary vectors. From \eqref{eqn:primal_step} we have 
\begin{align}\label{Primal state}
(x_{k+1}-x)^T\left(x_{k+1}-x_k+\gamma_k\left(F(x_k,\xi_k)+\left[\aj{\rho_k} f_{j_k}(x_k)+\lambda_k^{(j_k)} \right]_+\tilde \nabla f_{j_k}(x_k)\right)\right) \leq 0.
\end{align}
Using monotonicity of $F(\bullet)$ and Young's inequality, one can obtain
\begin{align}\label{F Monotone}
\nonumber &(x_{k+1}-x)^TF(x_k,\xi_k)\\
\nonumber&=(x_{k+1}-x_k)^TF(x_k)+(x_k-x)^TF(x_k)+(x_{k+1}-x)^Tw_k\\
\nonumber&\geq -\tfrac{1}{\aj{8}\gamma_k}\|x_{k+1}-x_k\|^2-\aj{2}\gamma_k\|F(x_k)\|^2+(x_k-x)^TF(x)+(x_{k+1}-x)^Tw_k\\
&\geq -\tfrac{1}{\aj{8}\gamma_k}\|x_{k+1}-x_k\|^2-\aj{2}\gamma_kC_F^2+(x_k-x)^TF(x)+(x_{k+1}-x)^Tw_k.
\end{align}
Similarly from Remark \ref{rem:stoch_err_subg} and convexity of $\Phi_\aj{\rho_k}(\bullet,\lambda_k)$, then we have
\begin{align}\label{convexity fee}
\nonumber&(x_{k+1}-x)^T\zz{\left[\rho_k f_{j_k}(x_k)+\lambda_k^{(j_k)} \right]_+\tilde \nabla f_{j_k}(x_k)}\\
\nonumber&=(x_{k+1}-x_k)^T\aj{\tilde{\nabla}_x} \Phi_\aj{\rho_k}(x_k,\lambda_k)+(x_k-x)^T\aj{\tilde{\nabla}_x} \Phi_\aj{\rho_k}(x_k,\lambda_k)+(x_{k+1}-x)^T\delta_k\\
\nonumber&\geq -\tfrac{1}{\aj{8}\gamma_k}\|x_{k+1}-x_k\|^2-\aj{2}\gamma_k\|\aj{\tilde{\nabla}_x} \Phi_\aj{\rho_k}(x_k,\lambda_k)\|^2+ \Phi_\aj{\rho_k}(x_k,\lambda_k)- \Phi_\aj{\rho_k}(x,\lambda_k)\\
\nonumber&\quad+(x_{k+1}-x)^T\delta_k\\
\nonumber&\geq -\tfrac{1}{\aj{8}\gamma_k}\|x_{k+1}-x_k\|^2-\aj{4}\gamma_kC_f^2\left(\aj{\rho_k^2}\ze{D_f^2}+\tfrac{1}{J}\|\lambda_k\|^2\right)+ \Phi_\aj{\rho_k}(x_k,\lambda_k)- \Phi_\aj{\rho_k}(x,\lambda_k)\\
&\quad+(x_{k+1}-x)^T\delta_k .
\end{align}
We can also write 
\begin{align}\label{Multiplication vector}
(x_{k+1}-x)^T(x_{k+1}-x_k)=\tfrac{1}{2}\left(\|x_{k+1}-x\|^2-\|x_{k}-x\|^2+\|x_{k+1}-x_k\|^2\right).
\end{align}
\za{Using \eqref{F Monotone},\eqref{convexity fee} and \eqref{Multiplication vector} in \eqref{Primal state}, we have}
\begin{align}\label{eq:b1}
\nonumber&(x_k-x)^TF(x)+\Phi_\aj{\rho_k}(x_k,\lambda_k)- \Phi_\aj{\rho_k}(x,\lambda_k) \\\nonumber& \quad\leq \tfrac{1}{2\gamma_k}\left(\|x_{k}-x\|^2-\|x_{k+1}-x\|^2\aj{-{1\over 2}\|x_{k+1}-x_k\|^2}\right)+\aj{2}\gamma_kC_F^2\\
&\qquad+\aj{4}\gamma_kC_f^2\left(\aj{\rho_k^2}\ze{D_f^2}+\tfrac{1}{J}\|\lambda_k\|^2\right)+\underbrace{(x-x_{k+1})^T(w_k+\delta_k)}_{\mbox{term (a)}}.
\end{align}
Now we obtain an upper bound for term (a) in \eqref{eq:b1}. 
\begin{align*}
\nonumber&(x-x_{k+1})^T(w_k+\delta_k)\\
\nonumber &\quad =(x-x_{k})^T(w_k+\delta_k)+(x_k-x_{k+1})^T(w_k+\delta_k)\\
& \quad\leq  (x-v_k)^T(w_k+\delta_k)+(v_k-x_k)^T(w_k+\delta_k)+{1\over 4\gamma_k}\|x_k-x_{k+1}\|^2+\gamma_k\|w_k+\delta_k\|^2.
\end{align*}
From Lemma \ref{lem:error} we have that $(x-v_k)^T(w_k+\delta_k)\leq {1\over 4\gamma_k}\|x-v_k\|^2-{1\over 4\gamma_k}\|x-v_{k+1}\|^2+\gamma_k\|w_k+\delta_k\|^2$, hence the above inequality can be written as 
\begin{align*}
(x-x_{k+1})^T(w_k+\delta_k)&\leq {1\over 4\gamma_k}\|x-v_k\|^2-{1\over 4\gamma_k}\|x-v_{k+1}\|^2+{1\over 4\gamma_k}\|x_k-x_{k+1}\|^\za{2}\\&+(v_k-x_k)^T(w_k+\delta_k)+2\gamma_k\|w_k+\delta_k\|^2.
\end{align*}
Using the above inequality in \eqref{eq:b1}, we get
\begin{align}\label{bb1}
\nonumber&(x_k-x)^TF(x)+ \Phi_\aj{\rho_k}(x_k,\lambda_k)- \Phi_\aj{\rho_k}(x,\lambda_k) \leq \tfrac{1}{2\gamma_k}\left(\|x_{k}-x\|^2-\|x_{k+1}-x\|^2\right)\\ \nonumber
&+2\gamma_kC_F^2+4\gamma_kC_f^2\left(\aj{\rho_k^2}\ze{D_f^2}+\tfrac{1}{J}\|\lambda_k\|^2\right)+\tfrac{1}{4\gamma_k}\left(\|x-v_k\|^2-\|x-v_{k+1}\|^2\right)\\&+(v_k-x_k)^T(w_k+\delta_k)+2\gamma_k\|w_k+\delta_k\|^2.
\end{align}
\aj{
Using Lemmas \ref{lem:bound_phi} and \ref{lem:bound_product}, we can bound the left hand side of \eqref{bb1} from below and one can obtain the following. 
\begin{align*}
&(x_k-x)^TF(x)+{1\over J}\sum_{j=1}^J \lambda^{(j)}f_j(x_k)- \Phi_\aj{\rho_k}(x,\lambda_k)\\&
\quad\leq \tfrac{1}{2\gamma_k}\left(\|x_{k}-x\|^2-\|x_{k+1}-x\|^2\right)+\tfrac{1} {4\gamma_k}\left(\|x-v_k\|^2-\|x-v_{k+1}\|^2\right)\\&
\qquad +\tfrac{1}{2\aj{\rho_k}}\left(\|\lambda_{k}-\lambda\|{^2}-\|\lambda_{k+1}-\lambda\|{^2}\right)+\tfrac{1}{2\bar\tau_k}\left(\|\lambda-\zz{\bar v_{k}}\|{^2}-\|\lambda-\zz{\bar v_{k+1}}\|{^2}\right)+2\gamma_kC_F^2\\&
\qquad +4\gamma_kC_f^2\left(\aj{\rho_k^2}\ze{D_f^2}+\tfrac{1}{J}\|\lambda_k\|^2\right)+(v_k-x_k)^T(w_k+\delta_k)+2\gamma_k\|w_k+\delta_k\|^2\\& \qquad + \zal{ (\bar v_k-\lambda_k)^T \bar \sigma_k}+\tfrac{\bar\tau_k\afj{\|\bar \sigma_k\|^2}}{2}+\Delta_k,
\end{align*}
{where $\Delta_k$ is defined in Lemma \ref{lem:bound_phi}.}
}
\end{proof}

\afz{Now we show that the sequence of dual iterates generated by the proposed method is bounded.}

\begin{lemma}\label{prop:lambda_k_squared_bound}\em
\aj{Consider Algorithm \ref{alg:primal_dual_SVI}. Let Assumptions \ref{assum:problem} and \ref{assum:samples} hold. Let  $\rho_k=\tfrac{\rho}{\sqrt{(k+1)}\log(k+1)}$, $\gamma_k=\tfrac{\gamma}{\sqrt{(k+1)}\log(k+1)}$, $t_k=\bar\tau_k=\tfrac{1}{\sqrt{(k+1)}\log(k+1)}$ for any $k\geq 1$, where $\rho\gamma\leq \tfrac{1}{\ze{120}\rho\gamma C_f^2/J}$. Moreover, we define $\rho_0=\rho$, $\gamma_0=\gamma$ and $t_0=\bar \tau_0=1$. Then, there exists $B\geq0$ such that $\mathbb E[\|\lambda_{K}\|{^2}]\leq B$ for any $K\geq 0$.

}
\end{lemma}
\begin{proof} \aj{From Lemma \ref{lem:err_func} we have $(x_k-x^*)^TF(x^*)+J^{-1}f(x_k)^{T}\lambda^* \geq 0$. Also, since $f_j(x^*)\leq 0$ for all $j \in J$, we have $\Phi_\rho(x^*,\lambda_k) \leq 0$. In view of these relations, from Proposition \ref{prop:main_rec_ineq}, for $x:=x^*$ and $\lambda:=\lambda^*$ we obtain
\begin{align*}
&
0 \quad \leq \tfrac{1}{2\gamma_k}\left(\|x_{k}-x^*\|^2-\|x_{k+1}-x^*\|^2\right)+\tfrac{1} {4\gamma_k}\left(\|x^*-v_k\|^2-\|x^*-v_{k+1}\|^2\right)\\&
\qquad+\tfrac{1}{2\aj{\rho_k}}\left(\|\lambda_{k}-\lambda^*\|{^2}-\|\lambda_{k+1}-\lambda^*\|{^2}\right)+\tfrac{1}{2\bar\tau_k}\left(\|\lambda^*-\zz{\bar v_{k}}\|{^2}-\|\lambda^*-\zz{\bar v_{k+1}}\|{^2}\right)+2\gamma_kC_F^2\\&
\qquad+4\gamma_kC_f^2\left(\aj{\rho^2_k}\ze{D_f^2}+\tfrac{1}{J}\|\lambda_k\|^2\right) +(v_k-x_k)^T(w_k+\delta_k)  +2\gamma_k\|w_k+\delta_k\|^2 \\&
\qquad + \zal{(\bar v_k-\lambda_k)^T \bar \sigma_k}+\tfrac{\bar\tau_k\afj{\|\bar \sigma_k\|^2}}{2}-\tfrac{1}{J}\sum_{j\in J_k^+}\tfrac{\rho_k}{2}(f_j(x_k))^2-\tfrac{1}{J}\sum_{j\in J_k^-}\tfrac{(\lambda_k^{(j)})^2}{2\rho_k}.
\end{align*}
\aj{Multiplying both sides by $t_k$ and  \zz{using the fact that} $\tfrac{t_k}{\rho_k}\geq \tfrac{t_{k+1}}{\rho_{k+1}}$, $\tfrac{t_k}{\gamma_k}\geq \tfrac{t_{k+1}}{\gamma_{k+1}}$, $t_k\geq t_{k+1}$, $\rho_k\geq \rho_{k+1}$, \zz{$\tfrac{t_k}{\bar\tau_k}\geq \tfrac{t_{k+1}}{\bar\tau_{k+1}}$}}, summing over \aj{$k=0,\ldots,T$, where $T\leq K$, and from $\|\lambda_{T+1}\|{^2} \leq 2\|\lambda_{T+1}-\lambda^*\|{^2}+2\|\lambda^*\|{^2}$ we obtain the following relation.}}
{\begin{align*}
&\zz{\tfrac{t_{T+1}}{4\rho_{T+1}}}\|\lambda_{T+1}\|{^2}\\& \quad \leq \tfrac{t_0}{2\gamma_0}\|x_{0}-x^*\|^2+\tfrac{t_0} {4\gamma_0}\|x^*-v_0\|^2+\tfrac{\zz{1}}{2\rho_0}\|\lambda_{0}-\lambda^*\|{^2}+\tfrac{t_0}{2\bar\tau_0}\|\lambda^*-\zz{\bar v_{0}}\|{^2}+\zz{\tfrac{t_{T+1}}{2\rho_{T+1}}}\|\lambda^*\|{^2}\\&
\qquad+\sum_{k=0}^Tt_k\gamma_k\left(2C_F^2+4C_f^2\left(\rho_k^2\ze{D_f^2}+\tfrac{1}{J}\|\lambda_k\|^2\right)\right)+\sum_{k=0}^Tt_k(v_k-x_k)^T(w_k+\delta_k)\\&
\qquad +2\sum_{k=0}^Tt_k\gamma_k\|w_k+\delta_k\|^2+ \sum_{k=0}^Tt_k\zal{ (\bar v_k-\lambda_k)^T \bar \sigma_k} +\sum_{k=0}^Tt_k\tfrac{\bar\tau_k\afj{\|\bar \sigma_k\|^2}}{2}-\sum_{k=0}^Tt_k\Delta_k.
\end{align*}
Taking expectation on the both sides and using Assumption \ref{assum:samples}(iv-v), \afj{Lemma \ref{lem:prop_stoch_err}} and the fact that $\mathbb E[\zal{ (\bar v_k-\lambda_k)^T \bar \sigma_k}]=\mathbb E[\Delta_k]=0$, we get  
\begin{align*}
&\tfrac{t_{T+1}}{4\rho_{T+1}}\mathbb E[\|\lambda_{T+1}\|{^2}]\\& \quad \leq \tfrac{t_0}{2\gamma_0}\|x_{0}-x^*\|^2+\tfrac{t_0} {4\gamma_0}\|x^*-v_0\|^2+\tfrac{1}{2\rho_0}\|\lambda_{0}-\lambda^*\|{^2}+\tfrac{t_0}{2\bar\tau_0}\|\lambda^*-\zz{\bar v_{0}}\|{^2}+\tfrac{t_{T+1}}{2\rho_{T+1}}\|\lambda^*\|{^2}\\&
\qquad+\sum_{k=0}^Tt_k\gamma_k\left(2C_F^2+4C_f^2\left(\rho_k^2\ze{D_f^2}+\tfrac{1}{J}\mathbb E[\|\lambda_k\|^2]\right)\right)\\&
\qquad +\zz{2}\sum_{k=0}^Tt_k\gamma_k(2\nu^2+\ze{4}C_f^2\rho_k^2D_f^2+\ze{4}C_f^2\mathbb E[\tfrac{\|\lambda_k\|^2}{J}])+\sum_{k=0}^Tt_k\tfrac{\bar\tau_kD_f^2}{2},
\end{align*}
\afj{where we used part (a) of Lemma \eqref{lem:bound_product} and the definition of $\bar \sigma_k$, i.e., $\mathbb E[\|\bar \sigma_k\|^2]\leq {D_f^2}$}. Define $A_1= \tfrac{t_0}{2\gamma_0}\|x_{0}-x^*\|^2+\tfrac{t_0} {4\gamma_0}\|x^*-v_0\|^2+\tfrac{1}{2\rho_0}\|\lambda_{0}-\lambda^*\|{^2}+\tfrac{t_0}{2\bar\tau_0}\|\lambda^*-\zz{\bar v_{0}}\|{^2}$, $A_2=2C_F^2+4C_f^2\rho^2\ze{D_f^2}$, $A_3=2\nu^2+\ze{4}C_f^2\rho^2D_f^2$, then the above inequality can be written as follows, where we used the fact that $\rho_k=\tfrac{\rho}{(k+1)\log(k+1)}\leq \rho$.
\begin{align}\label{bound simple}
\nonumber\tfrac{t_{T+1}}{4\rho_{T+1}}\mathbb E[\|\lambda_{T+1}\|{^2}]& \leq A_1+\tfrac{t_{T+1}}{2\rho_{T+1}}\|\lambda^*\|{^2}+\sum_{k=0}^Tt_k\gamma_k\left(A_2+4C_f^2\tfrac{1}{J}\mathbb E[\|\lambda_k\|^2]\right)\\&
\quad +2\sum_{k=0}^Tt_k\gamma_k(A_3+\ze{\tfrac{4}{J}}C_f^2\mathbb E[\|\lambda_k\|^2])+\sum_{k=0}^Tt_k\tfrac{\bar\tau_kD_f^2}{2}.
\end{align}
Letting $T=-1$, one can easily show that $\mathbb E[\|\lambda_0\|^2]\leq B$. Now suppose $\mathbb E[\|\lambda_{T+1}\|^2]\leq B$ holds for all $T\in\{-1,0,\hdots,K-2\}$. We show that $\mathbb E[\|\lambda_{T+1}\|^2]\leq B$ for $T=K-1$. Multiplying both sides of \eqref{bound simple} by $\tfrac{4\rho_{T+1}}{t_{T+1}}$ and letting $T=K-1$, we get  
\begin{align*}
\mathbb E[\|\lambda_{K}\|{^2}]& \leq \tfrac{4\rho_{K}}{t_{K}}A_1+2\|\lambda^*\|{^2}+\tfrac{4\rho_{K}}{t_{K}}\sum_{k=0}^{K-1}t_k\gamma_k\left(A_2+4C_f^2B\tfrac{1}{J}\right)\\&
\quad +\tfrac{8\rho_{K}}{t_{K}}\sum_{k=0}^{K-1}t_k\gamma_k(A_3+\ze{\tfrac{4}{J}}C_f^2B)+\tfrac{4\rho_{K}}{t_{K}}\sum_{k=0}^{K-1}t_k\tfrac{\bar\tau_kD_f^2}{2}.
\end{align*}
From the fact that $\rho_k=\tfrac{\rho}{\sqrt{(k+1)}\log(k+1)}$, $\gamma_k=\tfrac{\gamma}{\sqrt{(k+1)}\log(k+1)}$, $t_k=\bar\tau_k=\tfrac{1}{\sqrt{(k+1)}\log(k+1)}$, one can show that $\tfrac{\rho_k}{t_k}=\rho$ and $\sum_{k=0}^{K-1}t_k\gamma_k\leq 3\gamma$. Therefore, we obtain 
\begin{align*}
\mathbb E[\|\lambda_{K}\|{^2}]& \leq 4\rho A_1+2\|\lambda^*\|{^2}+12\rho\gamma\left(A_2+4C_f^2B\tfrac{1}{J}\right)+24\rho\gamma(A_3+\ze{\tfrac{4}{J}}C_f^2B)+12\rho\tfrac{D_f^2}{2}\leq B,
\end{align*}
where in the last inequality we used the fact that $B=\max\left\{\|\lambda_0\|^2,\tfrac{4\rho A_1+2\|\lambda^*\|^2+12\rho\gamma A_2+24\rho\gamma A_3+12\rho D_f^2}{1-144\rho\gamma C_f^2/J}\right\}$ and $\rho\gamma\leq \tfrac{1}{144\rho\gamma C_f^2/J}$.
}\end{proof}

Now we are ready to state the convergence rates of Algorithm \ref{alg:primal_dual_SVI}.

\begin{theorem}[Convergence rate statements for Algorithm~\ref{alg:primal_dual_SVI}]\label{thm:rates}\em
Consider Algorithm \ref{alg:primal_dual_SVI}. Let Assumptions \ref{assum:problem} and \ref{assum:samples} hold. Let  $\rho_k=\tfrac{\rho}{\sqrt{(k+1)}\log(k+1)}$, $\gamma_k=\tfrac{\gamma}{\sqrt{(k+1)}\log(k+1)}$, $t_k=\bar\tau_k=\tfrac{1}{\sqrt{(k+1)}\log(k+1)}$ for all $k\geq 1$, where $\rho\gamma\leq \tfrac{1}{144\rho\gamma C_f^2/J}$. Moreover, we define $\rho_0=\rho$, $\gamma_0=\gamma$ and $t_0=\bar \tau_0=1$. Let us define $\bar x_K\triangleq \tfrac{\sum_{k=0}^K t_kx_k}{\sum_{k=0}^K t_k}$ for $K\geq 0$. Then, for any $K\geq 0$, we have \begin{align*}
   & \mathbb E\left[\sup_{x\in \mathcal X}\{F(x)^T(\bar x_K-x)\}\right]\leq \mathcal O\left(\log(K+1)/\sqrt{K+2}\right)\\
    & \mathbb E\left[J^{-1} \mathbf{1}^T[f(\bar x_K)]_+\right]\leq\mathcal O\left(\log(K+1)/\sqrt{K+2}\right).
\end{align*} 
\end{theorem}

\begin{proof}
Multiplying both sides of \eqref{prop1} by $t_k$, using the fact that$\tfrac{t_k}{\rho_k}\geq \tfrac{t_{k+1}}{\rho_{k+1}}$, $\tfrac{t_k}{\gamma_k}\geq \tfrac{t_{k+1}}{\gamma_{k+1}}$, $t_k\geq t_{k+1}$, $\rho_k\geq \rho_{k+1}$, $\tfrac{t_k}{\bar\tau_k}\geq \tfrac{t_{k+1}}{\bar\tau_{k+1}}$, and summing $k=0$ to $K$, we get  
\begin{align}\label{bound in proof} 
\nonumber&\sum_{k=0}^K t_k\left((x_k-x)^TF(x)+{1\over J}\sum_{j=1}^J \lambda^{(j)}f_j(x_k)- \Phi_\rho(x,\lambda_k)\right)\\ \nonumber&
\quad \leq \tfrac{t_0}{2\gamma_0}\|x_0-x\|^2+\tfrac{t_0}{4\gamma_0}\|v_0-x\|^2+\tfrac{t_0}{2\aj{\rho_0}}\|\lambda_0-\lambda\|^2\\\nonumber&
\quad +\overbrace{\tfrac{t_0}{2\bar \tau_0}+\sum_{k=0}^K 2t_k\gamma_kC_F^2+\sum_{k=0}^K 4t_k\gamma_kC_f^2\left(\aj{\rho^2_k}\ze{D_f^2}+\tfrac{1}{J}\|\lambda_k\|^2\right)}^{\footnotesize\mbox {term (a)}}\\&
\qquad +\underbrace{\sum_{k=0}^K t_k \left((v_k-x_k)^T(w_k+\delta_k)+2\gamma_k\|w_k+\delta_k\|^2+ \zal{ (\bar v_k-\lambda_k)^T \bar \sigma_k}+\tfrac{\bar\tau_k\afj{\|\bar \sigma_k\|^2}}{2}+\Delta_k\right)}_{\footnotesize\mbox{term (b)}}.
\end{align}
Let right-hand side of \eqref{bound in proof} denoted by $C(x,\lambda)$. Dividing both sides of the above inequality by $\sum_{k=0}^K t_k$ and invoking the definition of $\bar{x}_k$, we get 
\begin{align*}
(\bar x_K-x)^TF(x)+{1\over J}\sum_{j=1}^J \lambda^{(j)}f_j(\bar x_K)- \Phi_\rho(x,\bar \lambda_k)\leq {1\over \sum_{k=0}^K t_k} C(x,\lambda),
\end{align*}
where in the left-hand side we used Jensen's inequality and the fact that $\Phi_\rho$ is concave with respect to $\lambda$.

Since $\bar x_K\in X$,  from Lemma \ref{lem main} (i) we have $J^{-1} \mathbf{1}^T[f(\bar x_K)]_+ \leq {C(x^*,\tilde \lambda)} $, where $\tilde \lambda$ is defined in Lemma \ref{lem main}. taking expectation on both side and using definition of $C(x,\lambda)$ in \eqref{bound in proof}, Lemmas \ref{lem:prop_stoch_err} \ref{prop:lambda_k_squared_bound}, Assumption \ref{assum:samples} (iv-v), the fact that $\mathbb E[\zal{(\bar v_k-\lambda_k)^T \bar \sigma_k}]=\mathbb E[\Delta_k]=0$ \afj{and $\mathbb E[\|\bar \sigma_k\|^2]\leq {D_f^2}$}, we obtain  

\begin{align}\label{bt1}
\nonumber\mathbb E\left[J^{-1} \mathbf{1}^T[f(\bar x_K)]_+\right]&\leq  {1\over \sum_{k=0}^K t_k} \Big[\tfrac{t_0}{2\gamma_0}\|x_0-x^*\|^2+\tfrac{t_0}{4\gamma_0}\|v_0-x^*\|^2+\tfrac{t_0}{2\aj{\rho_0}}\|\lambda_0-\tilde \lambda\|^2\\ \nonumber 
&+\tfrac{t_0}{2\bar \tau_0}+\sum_{k=0}^K 2t_k\gamma_kC_F^2+\sum_{k=0}^K 4t_k\gamma_kC_f^2\aj{ \rho^2_k}\ze{D_f^2}+ \sum_{k=0}^K 4t_k\gamma_kC_f^2 \tfrac{1}{J}B\\&+\sum_{k=0}^K t_k \left(2\gamma_k(2\nu^2+\ze{4}C_f^2\rho_k^2D_f^2+\ze{\tfrac{4}{J}}C_f^2B)+\tfrac{\bar\tau_kD_f^2}{2}\right)\Big].
\end{align}

Moreover, from Lemma \ref{lem main} (ii) we have $\sup_{x\in \mathcal X}\{F(x)^T(\bar x_K-x)\}\leq {\sup_{x\in \mathcal X}\{C(x,0)\}}$. By taking conditional expectation and then, unconditional expectation on both sides and using the fact that term (a) and term (b) in the definition of $C(x,\lambda)$ do not depend on $x$, we obtain  
\begin{align}\label{bt2}
\nonumber&\mathbb E\left[\sup_{x\in \mathcal X}\{F(x)^T(\bar x_K-x)\}\right]\\
\nonumber&\quad \leq  {1\over \sum_{k=0}^K t_k} \Big[\sup_{x\in \mathcal X}\left\{\tfrac{t_0}{2\gamma_0}\|x_0-x\|^2+\tfrac{t_0}{4\gamma_0}\|v_0-x\|^2+\tfrac{t_0}{2\aj{\rho_0}}\|\lambda_0\|^2\right\}+\tfrac{t_0}{2\bar \tau_0}\\ \nonumber
&\qquad+\sum_{k=0}^K 2t_k\gamma_kC_F^2+\sum_{k=0}^K 4t_k\gamma_kC_f^2\aj{ \rho^2_k}\ze{D_f^2}+ \sum_{k=0}^K 4t_k\gamma_kC_f^2 \tfrac{1}{J}B\\&\qquad+\sum_{k=0}^K t_k \left(2\gamma_k(2\nu^2+\ze{4}C_f^2\rho_k^2D_f^2+\ze{\tfrac{4}{J}}C_f^2B)+\tfrac{\bar\tau_kD_f^2}{2}\right)\Big].
\end{align}
From $\rho_k=\tfrac{\rho}{\sqrt{(k+1)}\log(k+1)}$, $\gamma_k=\tfrac{\gamma}{\sqrt{(k+1)}\log(k+1)}$, $t_k=\bar\tau_k=\tfrac{1}{\sqrt{(k+1)}\log(k+1)}$, and the facts that  $\rho_0=\rho$, $\gamma_0=\gamma$, $t_0=\bar \tau_0=1$, one can show that $\sum_{k=0}^K t_k\gamma_k\leq 3\gamma$ and similarly $\sum_{k=0}^K t_k\bar \tau_k\leq 3$, also $\sum_{k=0}^K t_k\geq \tfrac{1}{\log(K+1)}\int_{1}^{K+1}\tfrac{1}{\sqrt{x+1}}dx=\tfrac{2(\sqrt{K+2}-\sqrt{2})}{\log(K+1)}$. Therefore, we obtain that $\mathbb E\left[J^{-1} \mathbf{1}^T[f(\bar x_K)]_+\right]\leq\mathcal O(\log(K+1)/\sqrt{(K+2)})$ and similarly $\mathbb E\left[\sup_{x\in \mathcal X}\{F(x)^T(\bar x_K-x)\}\right]\leq \mathcal O(\log(K+1)/\sqrt{(K+2)})$.
\end{proof} 
Notably, the rate statements in Theorem~\ref{thm:rates} are in a mean sense, for both the dual gap function and the infeasibility metric. The latter quantifies the violation of the explicit functional constraints. A natural question is whether we can guarantee the convergence of the infeasibility metric to zero in an almost sure sense. This is partially addressed in the following result.

\begin{corollary}
    Consider Theorem~\ref{thm:rates}. There exists a subsequence of $\{\bar{x}_k\}$ along which, the infeasibility metric $\mathbf{1}^T[f(\bar x_K)]_+$ converges to zero almost surely.
\end{corollary}
   
\begin{proof}
From Theorem~\ref{thm:rates}, we have $\lim_{K \to \infty}\mathbb E\left[ \mathbf{1}^T[f(\bar x_K)]_+\right] =0$. Invoking Fatou’s lemma and noting that $  \mathbf{1}^T[f(\bar x_K)]_+ \geq 0$, we obtain 
$$\liminf_{K \to \infty}\  \mathbf{1}^T[f(\bar x_K)]_+  =0 \qquad \hbox{almost surely.}$$
Further, the sequence $\{\bar{x}_k\}$ is bounded, due to the projection onto the compact set $X$ in Algorithm~\ref{alg:primal_dual_SVI}. From the continuity of $f$, it follows that one of the (random) accumulation points of $\{\bar{x}_k\}$ must be a feasible point with respect to the explicit functional constraints almost surely.
\end{proof}

\section{Conclusion}\label{sec:conc}
In this paper, we consider stochastic variational inequality (VI) problems with a monotone mapping and a set that is characterized in terms of explicit functional constraints. Motivated by the absence of convergence rate statements for solving this class of problems, we develop a randomized Lagrangian stochastic approximation method where at each iteration the primal and dual variables are updated recursively. Our main contribution is to show that the existing convergence rates for nonlinearly constrained stochastic optimization problems can be extended to the stochastic VI regime. This is indeed promising and implies that the Lagrangian duality theory can be employed with provable guarantees for several important classes of problems that can be formulated as a stochastic VI. In particular, this work provides convergence speed guarantees for computing a Nash equilibrium in stochastic Nash games where each player may be associated with many hard-to-project constraints. 

\section{Appendix}\label{sec:app}
\subsection{Proof of Lemma \ref{lem:err_func}}
\begin{proof}
Invoking \zal {Proposition }\ref{assum:kkt} and taking into account that $\mathcal{N}_X(x^*) = \partial \mathcal{I}_X(x^*)$, we have that $x^* \in X$ solves the following augmented variational inequality problem
$
\mbox{VI}\left(X,F+J^{-1}\nabla f^T\lambda^*\right),
$ 
that is parameterized by $J$ and $\lambda^*$. This implies that 
\begin{align}\label{eqn:AVI_ineq}
\left(F(x^*) +J^{-1}\nabla f(x^*)^T\lambda^*\right)^T(x-x^*) \geq 0, \qquad \hbox{for all } x \in X.  
\end{align}
From the convexity of function $f_j$ for all $j \in [J]$ and that $\lambda_j \geq 0$, we have 
\begin{align*}
\fy{\lambda_j^*\left(f_j(x)-f_j(x^*)\right)}\geq  \lambda_j^*\nabla f_j(x^*)^T(x-x^*).
\end{align*}

Summing the preceding relation over $j \in [J]$ and recalling the definition of the mapping $f(x)$, we obtain
\begin{align*}  \left(f(x) -f(x^*)\right)^T\lambda^*  \geq \left(\nabla f(x^*)^T\lambda^*\right)^T(x-x^*).
\end{align*}
Invoking \zal{Proposition } \ref{assum:kkt} (ii) we obtain $f(x)^T\lambda^*   \geq \left(\nabla f(x^*)^T\lambda^*\right)^T(x-x^*)$. From the preceding relation and \eqref{eqn:AVI_ineq} we obtain $F(x^*) ^T(x-x^*) +J^{-1}f(x)^T\lambda^*  \geq 0$ for all $x \in X$. 
\end{proof}
\subsection{Proof of Lemma \ref{lem main}}
\begin{proof} (i) Note that $x^*$ is a feasible point to problem \eqref{prob:SVI_nlp} with respect to the set $\mathcal{X}$, i.e., $x^* \in \mathcal{X}$. Also, note that ${\hat \lambda \geq 0}$. From the definition of $\Phi_\rho$, we have that $\Phi_\rho(x^*,{\hat \lambda})\leq 0$. Let $x:=x^*$ in \eqref{ineq:decomp_lemma_assump_ineq}. Then we have
\begin{align}\label{ineq:decomp_proof_ineq1}
{F(x^*) ^T}(\hat x-x^*) +J^{-1}f(\hat x)^T\lambda  \leq {C(x^*,\lambda)}.
\end{align}
Also, from Lemma \ref{lem:err_func} and that $\hat{x} \in X$ we have 
\begin{align*}
{0 \leq  F(x^*) ^T(\hat{x}-x^*) +J^{-1}f(\hat{x})^T\lambda^*.}
\end{align*}
The preceding relation and that $\lambda^* \geq 0$ imply that 
\begin{align*}
0 \leq  {F(x^*) ^T}(\hat x-x^*) +J^{-1}[f(\hat x)]_+^T\lambda^*.
\end{align*}
Summing the preceding relation and \eqref{ineq:decomp_proof_ineq1} and rearranging the terms, we obtain 
\begin{align}\label{ineq:decomp_proof_ineq3}
 J^{-1}f(\hat x)^T\lambda-J^{-1}[f(\hat x)]_+^T\lambda^* \leq {C(x^*,\lambda)}.
\end{align}
Let us choose $\lambda_j:=1+\lambda^*_j$ if $f_j(\hat x)>0$, and $\lambda_j:=0$ otherwise for all $j \in [J]$. Then, we obtain the desired relation in (i). 

\noindent (ii) Let $\lambda=0$ in \eqref{ineq:decomp_lemma_assump_ineq} and note that $\Phi_\rho(x,\hat\lambda)\leq 0$ for all $x\in \mathcal X$. We have $F(x)^T(\hat x-x)\leq C(x,0)$ for all $x\in \mathcal X$. Taking supremum from the both sides, we obtain desired results in (ii).
\end{proof}
\subsection{Proof of Lemma \ref{lem:prop_stoch_err}}
\begin{proof}
The relations in part (i) hold as a consequence of Assumption \ref{assum:samples}. To show $\mathbb{E}[\delta_k \mid \mathcal{F}_k]=0$, we can write
\begin{align*}
\mathbb{E}[\delta_k \mid \mathcal{F}_k] &=\mathbb{E}\left[\left[{\rho_k} f_{j_k}(x_k)+\lambda_k^{(j_k)} \right]_+\tilde \nabla f_{j_k}(x_k)-\tfrac{1}{J}\textstyle\sum_{j=1}^J	 \left[{\rho_k} f_{j}(x_k)+\lambda_k^{(j)} \right]_+\tilde \nabla f_{j}(x_k)\mid \mathcal{F}_k\right]\\
& =\tfrac{1}{J}\textstyle\sum_{j=1}^J	 \left[{\rho_k} f_{j}(x_k)+\lambda_k^{(j)} \right]_+\tilde \nabla f_{j}(x_k)-\tfrac{1}{J}\textstyle\sum_{j=1}^J	 \left[{\rho_k} f_{j}(x_k)+\lambda_k^{(j)} \right]_+\tilde \nabla f_{j}(x_k) = 0,
\end{align*}
where the last inequality is implied from the assumption that $j_k$ is uniformly drawn from the set $[J]$. Next, we derive the bound on $\mathbb{E}[\|\delta_k\|^2\mid \mathcal{F}_k]$. We have
\begin{align*}
\mathbb{E}[\|\delta_k\|^2\mid \mathcal{F}_k]& = \mathbb{E}\left[\left\| \left[{\rho_k} f_{j_k}(x_k)+\lambda_k^{(j_k)} \right]_+\tilde \nabla f_{j_k}(x_k)\right\|^2\mid \mathcal{F}_k\right] + \left\|\tfrac{1}{J}\textstyle\sum_{j=1}^J	 \left[{\rho_k} f_{j}(x_k)+\lambda_k^{(j)} \right]_+\tilde \nabla f_{j}(x_k)\right\|^2\\
&-2\mathbb{E}\left[ \left[{\rho_k} f_{j_\za{k}}(x_k)+\lambda_k^{(j_\za{k})} \right]_+\tilde \nabla f_{j_\za{k}}(x_k)\mid \mathcal{F}_k\right]^T\left(\tfrac{1}{J}\textstyle\sum_{j=1}^J	 \left[{\rho_k} f_{j}(x_k)+\lambda_k^{(j)} \right]_+\tilde \nabla f_{j}(x_k)\right)\\
&= \tfrac{1}{J}\textstyle\sum_{j=1}^J\left\| \left[{\rho_k} f_{j}(x_k)+\lambda_k^{(j)} \right]_+\tilde \nabla f_{j}(x_k)\right\|^2 - \left\|\tfrac{1}{J}\textstyle\sum_{j=1}^J	 \left[{\rho_k} f_{j}(x_k)+\lambda_k^{(j)} \right]_+\tilde \nabla f_{j}(x_k)\right\|^2.
\end{align*}
Dropping the non-negative term in the preceding relation and invoking Remark \ref{rem:bounds}, we obtain 
\begin{align*}
\mathbb{E}[\|\delta_k\|^2\mid \mathcal{F}_k]&\leq 
 \tfrac{1}{J}\textstyle\sum_{j=1}^J\left\| \left[{\rho_k} f_{j}(x_k)+\lambda_k^{(j)} \right]_+\tilde \nabla f_{j}(x_k)\right\|^2 = \tfrac{1}{J}\textstyle\sum_{j=1}^J\left[{\rho_k} f_{j}(x_k)+\lambda_k^{(j)} \right]_+^2\left\| \tilde \nabla f_{j}(x_k)\right\|^2\\
 &\leq \tfrac{C_f^2}{J}\textstyle\sum_{j=1}^J\left({\rho_k} f_{j}(x_k)+\lambda_k^{(j)} \right)^2\leq \tfrac{2C_f^2}{J}\textstyle\sum_{j=1}^J\left(\rho_k^2D_f^2+\left(\lambda_k^{(j)}\right)^2\right) = 2C_f^2\left(\rho_k^2D_f^2+\tfrac{\|\lambda_k\|^2}{J}\right).
\end{align*}
\end{proof}
\subsection{Proof of Lemma \ref{lem:error}}
\begin{proof}
From the update rule of $v_{k+1}$,  \zal{we know}  $\sigma_k={1\over \tau_k}(v_{k+1}-v_k)$, hence we have that 
\begin{align*}
\zal{ \sigma^T_{k}(x-v_k)}&=\zal{ \sigma^T_{k}(x-v_{k+1})}+\zal{ \sigma^T_k (v_{k+1}-v_k)}\\
&\leq {1\over 2\tau_k}\|x-v_k\|^2-{1\over 2\tau_k}\|x-v_{k+1}\|^2-{1\over 2\tau_k}\|v_{k+1}-v_k\|^2+\zal{ \sigma^T_k(v_{k+1}-v_k)}\\
&\leq  {1\over 2\tau_k}\|x-v_k\|^2-{1\over 2\tau_k}\|x-v_{k+1}\|^2+{\tau_k\over 2}\|\sigma_k\|^2.
\end{align*}
\zal{first inequality is obtain from three points inequality.}
\end{proof}

\subsection{Proof of Lemma \ref{lem:bound_phi}}
\ze{\begin{proof}
From the fact that $\lambda_{k+1}-\lambda_k=J\rho_k e_{j_k}\odot \nabla_{\lambda} \Phi_{\rho_k}(x_k,\lambda_k)$, one can get the following:
\begin{align*}
\nonumber\tfrac{1}{\rho_k}\zal{ (\lambda_k-\lambda)^T (\lambda_{k+1}-\lambda_k)} &= \zal{ (\lambda_k-\lambda)^T (\nabla_\lambda\Phi_{\rho_k}(x_k,\lambda_k))}\\
&\quad+\zal{ (\lambda_k-\lambda)^T (Je_{j_k}\odot \nabla_{\lambda} \Phi_{\rho_k}(x_k,\lambda_k)-\nabla_\lambda\Phi_{\rho_k}(x_k,\lambda_k))}.
\end{align*}
\zal{also by knowing} that $ \tfrac{1}{\rho_k}\zal{ (\lambda_k-\lambda)^T (\lambda_{k+1}-\lambda_k)} =\tfrac{1}{2\rho_k}\left(\|\lambda_{k+1}-\lambda\|^2-\|\lambda_{k}-\lambda\|^2-\|\lambda_{k+1}-\lambda_k\|^2\right)$ and using previous equality one can obtain:

\begin{align}\label{lambda2}
\nonumber\tfrac{1}{2\rho_k}\|\lambda_{k+1}-\lambda\|^2 &= 
\tfrac{1}{2\rho_k}\|\lambda_{k}-\lambda\|^2+\tfrac{1}{2\rho_k}\|\lambda_{k+1}-\lambda_k\|^2+\zal{ (\lambda_k-\lambda)^T (\nabla_\lambda\Phi_{\rho_k}(x_k,\lambda_k))}\\
&\quad +\zal{ (\lambda_k-\lambda)^T (Je_{j_k}\odot \nabla_{\lambda} \Phi_{\rho_k}(x_k,\lambda_k)-\nabla_\lambda\Phi_{\rho_k}(x_k,\lambda_k))}
\end{align}
Using \eqref{lambda2}, one can easily show that: 
\begin{align}\label{LL}
&\nonumber-\Phi_{\rho_k}(x_k,\lambda_k)+{1\over J}\sum_{j=1}^J \lambda^{(j)}f_j(x_k)+\tfrac{1}{2\rho_k}\|\lambda_{k+1}-\lambda\|^2 \\
\nonumber&=-\Phi_{\rho_k}(x_k,\lambda_k)+{1\over J}\sum_{j=1}^J \lambda^{(j)}f_j(x_k)+ 
\tfrac{1}{2\rho_k}\|\lambda_{k}-\lambda\|^2+\tfrac{1}{2\rho_k}\|\lambda_{k+1}-\lambda_k\|^2+\zal{ (\lambda_k-\lambda)^T (\nabla_\lambda\Phi_{\rho_k}(x_k,\lambda_k))}\\
&\quad +\zal{ (\lambda_k-\lambda)^T (Je_{j_k}\odot \nabla_{\lambda} \Phi_{\rho_k}(x_k,\lambda_k)-\nabla_\lambda\Phi_{\rho_k}(x_k,\lambda_k))}.
\end{align}\\
From definition of $\Phi_{\rho_k}(x_k,\lambda_k)$, $J^+_k$ and $J_k^-$ we have  :
\begin{align}\label{def phi}
\Phi_{\rho_k}(x_k,\lambda_k)=\tfrac{1}{J}\Big[\sum_{j\in J_k^+}(\tfrac{\rho_k}{2}(f_j(x_k))^2+\lambda_k^{(j)}f_j(x_k))-\sum_{j\in J_k^-}\tfrac{(\lambda_k^{(j)})^2}{2\rho_k}\Big]. 
\end{align}
Using \eqref{LL}, \eqref{def phi} and the fact that $\nabla_{\lambda} \Phi_\rho(x,\lambda)=\tfrac{1}{J}\Big[\max(\tfrac{-\lambda^{(j)}}{\rho},f_j(x))\Big]_{j=1}^{J}$, the following holds:
\begin{align}\label{lemma5 proof}
&\nonumber-\Phi_{\rho_k}(x_k,\lambda_k)+{1\over J}\sum_{j=1}^J \lambda^{(j)}f_j(x_k)+\tfrac{1}{2\rho_k}\|\lambda_{k+1}-\lambda\|^2 \\
\nonumber&=-\tfrac{1}{J}\sum_{j\in J_k^+}\tfrac{\rho_k}{2}(f_j(x_k))^2+\tfrac{1}{J}\sum_{j\in J_k^-}\Big[\tfrac{(\lambda_k^{(j)})^2}{2\rho_k}+\lambda^{(j)}f_j(x_k)+(\lambda_k^{(j)}-\lambda^{(j)})(\tfrac{-\lambda_k^{(j)}}{\rho_k})\Big]\\
\nonumber&\quad+\tfrac{1}{2\rho_k}\|\lambda_{k}-\lambda\|^2+\tfrac{1}{2\rho_k}\|\lambda_{k+1}-\lambda_k\|^2+\zal{ (\lambda_k-\lambda)^T (\nabla_\lambda\Phi_{\rho_k}(x_k,\lambda_k))}\\
&\nonumber\quad+\zal{ (\lambda_k-\lambda)^T (Je_{j_k}\odot \nabla_{\lambda} \Phi_{\rho_k}(x_k,\lambda_k)-\nabla_\lambda\Phi_{\rho_k}(x_k,\lambda_k))}\\
&= -\tfrac{1}{J}\sum_{j\in J_k^+}\tfrac{\rho_k}{2}(f_j(x_k))^2-\tfrac{1}{J}\sum_{j\in J_k^-}(\tfrac{(\lambda_k^{(j)})^2}{2\rho_k}-\lambda^{(j)}(f_j(x_k)+\tfrac{\lambda_k^{(j)}}{\rho_k}))\\
\nonumber&\quad+\tfrac{1}{2\rho_k}\|\lambda_{k}-\lambda\|^2+\tfrac{1}{2\rho_k}\|\lambda_{k+1}-\lambda_k\|^2+\zal{ (\lambda_k-\lambda)^T (Je_{j_k}\odot \nabla_{\lambda} \Phi_{\rho_k}(x_k,\lambda_k)-\nabla_\lambda\Phi_{\rho_k}(x_k,\lambda_k))}.
\end{align}
Note that $\lambda\geq 0$ and by definition $J_k^-$ it holds that $\lambda^{(j)}(f_j(x_k)+\tfrac{\lambda_k^{(j)}}{\rho_k})\leq 0$, so we conclude that
\begin{align}\label{bound lemma5}
&-\tfrac{1}{J}\sum_{j\in J_k^+}\tfrac{\rho_k}{2}(f_j(x_k))^2-\tfrac{1}{J}\sum_{j\in J_k^-}(\tfrac{(\lambda_k^{(j)})^2}{2\rho_k}-\lambda^{(j)}(f_j(x_k)+\tfrac{\lambda_k^{(j)}}{\rho_k}))\leq-\tfrac{1}{J}\sum_{j\in J_k^+}\tfrac{\rho_k}{2}(f_j(x_k))^2-\tfrac{1}{J}\sum_{j\in J_k^-}\tfrac{(\lambda_k^{(j)})^2}{2\rho_k}.
\end{align}
Hence we have the desired result by putting \eqref{bound lemma5} in \eqref{lemma5 proof}.
\end{proof}}

\subsection{Proof of lemma \ref{lem:bound_product}}
\begin{proof}
(a) From definition of $\nabla_\lambda\Phi_\aj{\rho_k}$, using Assumption \ref{assum:problem} (ii) and the fact that $\lambda_k^{(j_k)}\geq 0$ for all $k$ and $j$, we have that $\|Je_{j_k}\odot \nabla_{\lambda} \Phi_\aj{\rho_k}(x_k,\lambda_k)\|^2=\left|\max \left(\tfrac{-\lambda_k^{(j_k)}}{\aj{\rho_k}},f_{j_k}(x_k)\right)\right|^2\leq D_f^2.$\\

\noindent (b) \zal{By} definition of  $\bar\sigma_k$ and $\bar v_k$ and using Lemma \ref{lem:error}, one can obtain the following. 
\begin{align*}
&\zal{ (\lambda-\lambda_k\pm \bar v_k)^T (\nabla_\lambda\Phi_\aj{\rho_k}(x_k,\lambda_k)-Je_{j_k}\odot \nabla_{\lambda} \Phi_\aj{\rho_k}(x_k,\lambda_k))}\\
&\quad = \zal{ (\bar v_k-\lambda_k)^T \bar \sigma_k}+\zal{ (\lambda -\bar v_k)^T \bar \sigma_k} \leq \zal{(\bar v_k-\lambda_k)^T \bar \sigma_k} +{1\over 2\bar\tau_k}\|\lambda-\bar v_k\|^2-{1\over 2\bar\tau_k}\|\lambda-\bar v_{k+1}\|+{\bar \tau_k\over 2}\|\bar \sigma_k\|^2.
\end{align*}
\end{proof}
\section{Acknowledgments}
 This work is supported in part by the National Science Foundation under CAREER Grant ECCS-1944500 and Grant ECCS-2231863, the
Office of Naval Research under Grant N00014-22-1-2757, the University of Arizona Research, Innovation \& Impact (RII) Funding, and the Arizona Technology and Research Initiative Fund (TRIF) for Innovative Technologies for the Fourth Industrial Revolution initiatives.
\bibliographystyle{siam}
\bibliography{biblio,ref_pairIG_v01_fy}

\end{document}